\documentclass[12pt]{article}
\usepackage{amsmath}
\usepackage{float}
\usepackage{amsthm}
\usepackage{amssymb}
\usepackage[utf8x]{inputenc}
\usepackage{caption}
\usepackage{graphics}
\usepackage{enumerate} 
\usepackage{algorithmic}
\usepackage{calrsfs}
\DeclareMathAlphabet{\pazocal}{OMS}{zplm}{m}{n}

\usepackage[Algorithm,ruled]{algorithm}
\usepackage{titlesec}
\usepackage{sectsty}
\titleformat*{\section}{\LARGE\bfseries}
\titleformat*{\subsection}{\Large\bfseries}
\titleformat*{\subsubsection}{\large\bfseries}
\sectionfont{\Huge}
\subsectionfont{\large}
\subsubsectionfont{\normalsize}
\usepackage{multicol}
\usepackage{lmodern}
\usepackage{marvosym} 
\usepackage{lipsum}
\usepackage{mwe}
\usepackage{caption}
\usepackage{subcaption} 
\newtheoremstyle{case}{}{}{}{}{}{:}{ }{}
\theoremstyle{case}

\newcommand{\be}{\begin{equation}}
\newcommand{\ee}{\end{equation}}
\newcommand{\ben}{\begin{eqnarray*}}
\newcommand{\een}{\end{eqnarray*}}
\newtheorem{examp}{\sc example}
\newtheorem{remk}{\sc remark}
\newtheorem{corol}{\sc corollary}
\newtheorem{lemma}{\sc lemma}
\newtheorem{theorem}{\sc theorem}
\newtheorem{defn}{\sc definition}
\newtheorem{prop}{\sc proposition}
\newcommand{\bt}{\begin{theorem}}
\newcommand{\et}{\end{theorem}}
\newcommand{\bl}{\begin{lemma}}
\newcommand{\el}{\end{lemma}}
\newcommand{\bed}{\begin{defn}}
\newcommand{\eed}{\end{defn}}
\newcommand{\brem}{\begin{remk}}
\newcommand{\erem}{\end{remk}}
\newcommand{\bex}{\begin{examp}}
\newcommand{\eex}{\end{examp}}
\newcommand{\bcl}{\begin{corol}}
\newcommand{\ecl}{\end{corol}}

\newcommand{\NI}{\noindent}

\theoremstyle{definition}
\theoremstyle{remark}

\numberwithin{equation}{section}
\numberwithin{theorem}{section}
\numberwithin{lemma}{section}

\begin{document}

\title{\large\bf\sc Necessary and sufficient conditions for a subclass of $P$-tensor}

\author{ 
R. Deb$^{a,1}$ and A. K. Das$^{b,2}$\\
\emph{\small $^{a}$Jadavpur University, Kolkata , 700 032, India.}\\	
\emph{\small $^{b}$Indian Statistical Institute, 203 B. T.
	Road, Kolkata, 700 108, India.}\\
%\emph{\small $^{1}$Email: rwitamjanaju@gmail.com}\\
%\emph{\small $^{1}$Email: aritradutta001@gmail.com}\\
\emph{\small $^{1}$Email: rony.knc.ju@gmail.com}\\
\emph{\small $^{2}$Email: akdas@isical.ac.in}\\
}

\date{}

\maketitle

\begin{abstract}
\NI In this article, we introduce the class $B$-Nekrasov tensor in the context of tensor complementarity problem. We study some tensor theoretic properties. We show that the class of B-Nekrasov tensor contains the class of Nekrasov $Z$-tensor with positive diagonal entries. We present a necessary and sufficient condition for a $B$-Nekrasov tensor. We show that the class of $P$-tensor contains the class of $B$-Nekrasov tensor.\\

\\

\noindent{\bf AMS subject classifications:} 
\end{abstract}
\footnotetext[1]{Corresponding author}

\section{Introduction}

\NI Given a positive integer $n,$ the set $\{1,2,...,n\}$ is denoted by $[n].$ A tensor of order $m$ and dimension $n$, $\mathcal{A}= (a_{i_1 i_2... i_m}) $ is a multidimensional array of entries $a_{i_1 i_2... i_m} \in \mathbb{C}$ where $i_l \in [n]$ with $l\in [m].$ The set of complex (real) tensors of order $m$ and dimension $n$ is denoted by $\mathbb{C}^{[m,n]} (\mathbb{R}^{[m,n]}).$ If under any permutation of the indices $i_1, i_2, ... ,i_m$ the entries $a_{i_1 i_2... i_m}$ are invariant then the tensor $\mathcal{A}= (a_{i_1 i_2... i_m}) \in \mathbb{C}^{[m,n]}$ is called a symmetric tensor. Tensors have applications in different areas of science and engineering. The common applications are in physics, quantum computing, spectral hypergraph theory, diffusion tensor imaging, image authenticity verification problem, optimization theory and many other areas. In optimization theory, the tensor complementarity problem (TCP) introduced by Song and Qi \cite{song2017properties}, is a subclass of the nonlinear complementarity problem, where the involved functions are special polynomials defined by a tensor.

\NI Given a real tensor $\mathcal{A}$ of order $m$ and dimension $n$ and an real vector $q\in\mathbb{R}^n$ the tensor complementarity problem is to look for $v\in \mathbb{R}^n $ such that  satisfying
	\begin{equation}\label{tensor comp equation}
	v\geq 0,\;\;  \mathcal{A}v^{m-1}+q \geq 0,\;\;  v^T (\mathcal{A}v^{m-1}+q)  = 0.
	\end{equation}
This problem is denoted by TCP$(\mathcal{A}, q)$ and the solution set of TCP$(\mathcal{A}, q)$ is denoted by SOL$(\mathcal{A}, q).$
The tensor complementarity problem arises in optimization theory, game theory and in other areas. A class of multi-person noncooperative game \cite{huang2017formulating}, hypergraph clustering problem and traffic equilibrium problem \cite{huang2019tensor} can be formulated as the tensor complementarity problem.

\NI The tensor complementarity problem can be viewed as one type of extension of the linear complementarity problem where the homogeneous polynomials involved in the function are constructed with the help of tensor.
\NI Given a real square matrix $A$ of order $n$ and an vector $q\in \mathbb{R}^n$, the linear complementarity problem \cite{cottle2009linear} is to find $z$ satisfying
\begin{equation}\label{linear comp equation}
 z\geq 0,\;\;  Az + q \geq 0,\;\; z^T (Az + q) = 0.
\end{equation}
The problem is denoted by LCP$(q,A)$ and the solution set of LCP$(q,A)$ is denoted by SOL$(q,A).$ It is important that large number of formulations not only enrich the linear complementarity problem but also generate different matrix classes along with their computational methods.
 For details see \cite{das2019some}, \cite{neogy2006some}, \cite{dutta2021some} \cite{dutta2022on}, \cite{neogy2013weak}, \cite{jana2018semimonotone}, \cite{neogy2005almost}, \cite{neogy2011singular}, \cite{jana2019hidden}, \cite{jana2021more}, \cite{neogy2009modeling}, \cite{das2017finiteness}, \cite{das2018invex}. For details of game theory see \cite{mondal2016discounted}, \cite{neogy2008mathematical}, \cite{neogy2008mixture}, \cite{neogy2005linear}, \cite{neogy2016optimization}, \cite{das2016generalized}, \cite{neogy2011generalized} and for details of QMOP see \cite{mohan2004note}. Even matrix classes arise during the study of Lemke's algorithm as well as principal pivot transform. For details see \cite{mohan2001classes}, \cite{mohan2001more} \cite{neogy2005principal}, \cite{das2016properties}, \cite{neogy2012generalized}, \cite{jana2019hidden}, \cite{jana2021iterative}, \cite{jana2021more}, \cite{jana2018processability}. 

\NI Orea and Pe$\tilde{n}$a \cite{orera2019accurate} introduced Nekrasov $Z$-matrix and studied the methods to compute its inverse with high relative accuracy. Esnaola and Pe$\tilde{n}$a \cite{garcia2016b} introduced $B$-Nekrasov matrix and studied the error bounds for the linear complementarity problem involving a $B$-Nekrasov matrix.

\NI In tensor complementarity theory, some special classes of matrices are extended to higher order structured tensors. In recent years many structured tensors are developed and studied well in context of TCP. The class of $P(P_0)$-tensor, $Z$-tensor, $M$-tensor, $H$-tensor, Nekrasov tensor and many other tensor classes are introduced and studied in the context of tensor complementarity problem. Zhang et al. \cite{zhang2014m} introduce $Z$-tensor as well as $M$-tensor in context of tensor complementarity problem and show that for a $Z$-tensor the tensor complementarity problem has least element property. $H$-tensor is introduced by Ding et al. \cite{ding2013m} and studied thoroughly in literature for its importance in tensor complementarity problem. %Song and Qi \cite{song2014properties} introduce $B$-tensor and study some tensor theoretic properties as well as properties related to tensor complementarity problem.
The Nekrasov tensors is introduced by Zhang and Bu \cite{zhang2018nekrasov}. They have proved that a Nekrasov tensor is a $H$-tensor. The positive definitene tensors have many important applications in multivariate network realizability analysis, automatic control, medical imaging and so on \cite{basser2002diffusion}, \cite{ni2008eigenvalue}, \cite{ qi2010higher}. Song and Qi \cite{song2014properties} introduce $P$-tensor. Severl properties of $P$-tensors are studied in context of tensor complementarity problem. For details see \cite{bai2016global}, \cite{ding2018p}, \cite{yuan2014some}. Qi and Luo \cite{qi2017tensor} has studied Several important results related to positive definite tensors. The positive definiteness of a homogeneous polynomial is equivalent to the positive definiteness of the symmetric tensor associated with it \cite{kannan2015some}. Since an even order real symmetric nonsingular $H$-tensor with positive diagonal entries is positive definite \cite{li2014criterions}, nonsingular H-tensors play an important role in checking the positive definiteness of even order real symmetric tensors. % Deb and Das \cite{} introduced Nekrasov $Z$-tensor and studied its properties.

\NI In this paper, we introduce $B$-Nekrasov tensor. We discuss necessary and sufficient condition for $B$-Nekrasov tensor. We prove that a $B$-Nekrasov tensor is a $P$-tensor.

The paper is organised as follows. Section 2 contains some basic notations and results. In Section 3, we introduce $B$-Nekrasov tensor. We investigate the connection between Nekrasov $Z$-tensor and $B$-Nekrasov tensor. We propose a necessary and sufficient condition for $B$-Nekrasov tensor. We establish its connection with $P$-tensor.

\section{Preliminaries}

Here are some basic notations used in this paper. For any positive integer $n,$ let $\mathbb{R}^n$ denote the $n$-dimensional Euclidean space. Any vector $v\in \mathbb{C}^n$ is a column vector and $v^T$ denotes the row transpose of $v.$ Here $\mathbb{R}^n_+ =\{ v\in \mathbb{R}^n : v\geq 0 \}$, $\mathbb{R}^n_{+ +} =\{ v\in \mathbb{R}^n : v> 0 \}$. A diagonal matrix $D=[d_{ij}]_{n \times n}=diag(d_1, \; d_2,\; ..., \; d_n)$ is defined as $d_{ij}=\left \{ \begin{array}{ll}
	  d_i  &;\; \forall \; i=j, \\
	  0  &; \; \forall \; i \neq j.
	   \end{array}  \right.$

\begin{defn}\cite{cvetkovic2009new,bailey1969bounds,kolotilina2015some, garcia2014error}
A matrix $A=(a_{ij})\in \mathbb{C}^{[2,n]},$ is said to be a Nekrasov matrix if $|a_{ii}|> \Lambda_i(A), \; \forall\; i\in [n],$ where
\begin{center}
    $\Lambda_i(A) =\left\{\begin{array}{ll}
        \sum_{j=2}^n |a_{ij}| &, i=1,  \\
        \sum_{j=1}^{i-1} |a_{ij}| \frac{\Lambda_j(A)}{|a_{jj}|} + \sum_{j=i+1}^n |a_{ij}| &, i=2,3,...,n.
    \end{array} \right.$
\end{center}
\end{defn}

\begin{defn}\cite{mangasarian1976linear}
A matrix $A\in \mathbb{R}^{[2,n]},$ is said to be a $Z$-matrix if all its offdiagonal entries are nonpositive.
\end{defn}
	   
%\begin{defn}\cite{orera2019accurate}
%A matrix $A\in \mathbb{R}^{[2,n]}$ is said to be Nekrasov $Z$-matrix, if $A$ is a Nekrasov mtrix as well as a $Z$-matrix.
%\end{defn}

\begin{defn}\cite{garcia2016b}
Given a real matrix $A(a_{ij})\in \mathbb{R}^{[2,n]},$ we can write $A$ as $A= B^+ + C,$ where
\begin{center}
    $B^+ = \left(\begin{array}{ccc}
        a_{11}-r_1^+ & ... & a_{1n}-r_1^+ \\
        \vdots & \vdots & \vdots \\
        a_{n1}-r_n^+ & ... & a_{nn}-r_n^+
    \end{array} \right)$ and 
    $C = \left(\begin{array}{ccc}
        r_1^+ & ... & r_1^+ \\
        \vdots & \vdots & \vdots \\
        r_n^+ & ... & r_n^+
    \end{array} \right)$
\end{center}
with $r_i^+ = \max \{0, a_{i,j} | j\neq i\}.$ The matrix $A$ is said to be a $B$-Nekrasov matrix if $B^+$ is a Nekrasov $Z$-matrix whose diagonal entries are all positive.
\end{defn}

\NI An identity tensor of order $m,$ $\mathcal{I}=(\delta_{i_1 i_2... i_m})\in \mathbb{C}^{[m,n]}$ is defined as follows:
$\delta_{i_1 i_2... i_m}= \left\{
\begin{array}{ll}
	  1  &:\; i_1= ...= i_m \\
	  0  &:\; else
	   \end{array}
 \right. .$
Let the zero tensor be denoted by $\mathcal{O},$ where each entry of $\mathcal{O}$ is zero. For $\mathcal{A}\in \mathbb{C}^{[m,n]}$ and $v\in \mathbb{C}^n,\; \mathcal{A}v^{m-1}\in \mathbb{C}^n $ is a vector defined by
\[ (\mathcal{A}v^{m-1})_i = \sum_{i_2, ..., i_m =1}^{n} a_{i i_2 ... i_r} v_{i_2}  \cdots v_{i_m} , \mbox{   for all } i \in [n] \]
and $\mathcal{A}v^m\in \mathbb{C} $ is a scalar defined by
\[ v^T\mathcal{A}v^{m-1}= \mathcal{A}v^m = \sum_{i_1, i_2, ..., i_m =1}^{n} a_{i_1 i_2 ...i_m} v_{i_1} v_{i_2} \cdots v_{i_m}. \]

\NI The general product of tensors was introduced by Shao \cite{shao2013general}. Let $\mathcal{A}$ and $\mathcal{B}$ be two $n$ dimensional tensor of order $p \geq 2$ and $r \geq 1,$ respectively. The product $\mathcal{A} \mathcal{B}$ is an $n$ dimensional tensor $\mathcal{C}$ of order $((p-1)(r-1)) + 1$ with entries 
\[c_{j \beta_1 \cdots \beta_{p-1} } =\sum_{j_2, \cdots ,j_p \in [n]} a_{j j_2 \cdots j_p} b_{j_2 \beta_1} \cdots b_{j_p \beta_{p-1}},\] where $j \in [n]$, $\beta_1, \cdots, \beta_{p-1} \in [n]^{r-1}$

\begin{defn}\cite{song2016properties}
Given $\mathcal{A} \in \mathbb{R}^{[m,n]} $ and $q\in \mathbb{R}^n$, a vector $v$ is said to be (strictly) feasible solution of TCP$(\mathcal{A},q),$ if $v \geq 0\; (>0)$ and $\mathcal{A}v^{m-1}+q \geq 0\; (>0)$.
\end{defn}

\begin{defn}\cite{song2016properties}
Given $\mathcal{A} \in \mathbb{R}^{[m,n]} $ and $q\in \mathbb{R}^n$, TCP$(\mathcal{A},q)$ is said to be (strictly) feasible if a (strictly) feasible vector exists.
\end{defn}

\begin{defn}\cite{song2016properties}
Given $\mathcal{A} \in \mathbb{R}^{[m,n]} $ and $q\in \mathbb{R}^n$, TCP$(\mathcal{A},q)$ is said to be solvable if there exists a feasible vector $v$ satisfying $v^{T}(\mathcal{A}v^{r-1}+q)=0$ and $v$ is said to be a solution of the TCP$(\mathcal{A},q)$.
\end{defn}

\begin{defn}
\cite{song2015properties} A tensor $\mathcal{A}\in \mathbb{R}^{[m,n]} $ is said to be a $P_0(P)$-tensor, if for each $v\in \mathbb{R}^n \backslash \{0\}$, there exists an index $j\in [n]$ such that $v_j \neq 0$ and $v_j (\mathcal{A}v^{m-1})_j \geq 0\; (>0)$.
\end{defn}

%\begin{defn}\cite{huang2015q} 
%A tensor $\mathcal{A}\in \mathbb{R}^{[m,n]}$ is said to be a $Q$-tensor if the TCP$(\mathcal{A},q)$ is solvable for all $q\in \mathbb{R}^n $.
%\begin{center}
%    i.e., $v\geq 0, ~~~  \mathcal{A}v^{m-1} + q\geq 0, ~~~\mbox{and}~~ v^{T} (\mathcal{A}v^{m-1} + q) =0.$
%\end{center}
%\end{defn}

\NI The row subtensors are defined in \cite{shao2016some}. Here for the sake of convenience we denote $i$-th rowsubtensor of $\mathcal{A}$ by $\mathcal{R}_i(\mathcal{A}).$
\begin{defn}\cite{shao2016some}
For each $i$ the $i$th row subtensor of $\mathcal{A}\in \mathbb{C}^{[m,n]}$ is denoted by $\mathcal{R}_i(\mathcal{A})$ and its entries are given as $(\mathcal{R}_i(\mathcal{A}))_{i_2 ... i_m}=(a_{i i_2... i_m})$, where $i_l\in [n]$ and $2\leq l \leq n.$
\end{defn}

\begin{defn}\cite{zhang2018nekrasov}
Let $\mathcal{A}\in \mathbb{C}^{[m,n]}$ such that $\mathcal{A}=(a_{i_1 i_2 ... i_m}).$ Let for all $i\in [n],$
\begin{equation*}
    R_i(\mathcal{A})= \sum_{(i_2 ... i_m) \neq (i ... i)} |a_{i i_2 ... i_m}|.
\end{equation*}
The tensor $\mathcal{A}$ is said to be diagonally dominant (strict diagonally dominant) if $|a_{i, ..., i}| \geq (>) R_i(\mathcal{A})$ for all $i\in [n].$
\end{defn}

\begin{defn}\cite{zhang2018nekrasov}
Let $\mathcal{A}\in \mathbb{C}^{[m,n]}$ is called quasi-diagonally dominant tensor if there exists a diagonal matrix $D=diag(d_1, d_2, \cdots, d_n)$ such that $\mathcal{A} D$ is strictly diagonally dominant tensor, i.e.,
\begin{equation*}
  |a_{i i \cdots i}|d_i^{m-1} > |a_{i i_2 \cdots i_m}|d_{i_2} \cdots d_{i_m}, \;\;\; \forall\; i\in [n]. 
\end{equation*}
\end{defn}

\NI The concept of eigenvalues and eigenvectors for tensors was proposed by Qi \cite{qi2005eigenvalues} and Lim \cite{lim2005singular}. If a pair $(\lambda, v) \in \mathbb{C}\times (\mathbb{C}^n \backslash \ \{0\})$ satisfies the equation $\mathcal{A} v^{m-1} = \lambda v^{m-1},$ then $\lambda$ is said to be an eigenvalue of $\mathcal{A}$, and $v$ is said to be an eigenvector corresponding to $\lambda$ of $\mathcal{A}$, where $v^{[m−1]} = (v_1^{m−1},... , v_n^{m−1})^T.$ The quantity $\rho(\mathcal{A}) = \max\{|\lambda| \; : \; \lambda \text{ is an eigen value of } \mathcal{A}\}$ is said to be the spectral radius of $\mathcal{A}$.

\begin{defn}\cite{ding2013m}
A tensor is said to be $Z$ tensor if all its offdiagonal entries are nonpositive.
\end{defn}
\begin{defn}\cite{ding2013m}
A $Z$-tensor $\mathcal{A}$ is said to be an $M$-tensor if $\exists \; \mathcal{B}\; \geq \mathcal{O}$ and $s>0$ such that $\mathcal{A}= s \mathcal{I} - \mathcal{B}$, where $s\geq \rho(\mathcal{B})$. $\mathcal{A}$ is said to be a nonsingular $M$-tensor if $s>\rho(\mathcal{B}).$
\end{defn}

\begin{defn}\cite{ding2013m}
For a tensor $\mathcal{A}= a_{i_1 i_2 ... i_m}\in \mathbb{C}^{[m,n]},$ a tensor $\mathcal{C}(\mathcal{A})= c_{i_1 i_2 ... i_m}$ is said to be the comparison tensor of $\mathcal{A}$ if 
\begin{center}
    $c_{i_1 i_2 ... i_m}=\left\{\begin{array}{ll}
       |a_{i_1 i_2 ... i_m}|  & , \text{ if } (i_1, i_2, ... ,i_m) =(i,i, ..., i), \\
       -|a_{i_1 i_2 ... i_m}|  & , \text{ if } (i_1, i_2, ... ,i_m) \neq (i,i, ..., i).
    \end{array} \right.$
\end{center}
\end{defn}

\begin{defn}\cite{ding2013m}
A tensor $\mathcal{A}$ is said to be an $H$-tensor, if its comparison tensor is an $M$-tensor, and it is said to be a nonsingular $H$-tensor, if its comparison tensor is a nonsingular $M$-tensor.
\end{defn}

\begin{defn}
A tensor $\mathcal{A}\in \mathbb{R}^{[m,n]}$ is said to be a Nekrasov $Z$-tensor if $\mathcal{A}$ is a Nekrasov tensor as well as a $Z$-tensor.
\end{defn}

\begin{defn}\cite{zhang2018nekrasov}
For $\mathcal{A}=(a_{i_1 i_2 ... i_m}) \in \mathbb{C}^{[m,n]}, \; a_{i i ... i}\neq 0, \; \forall \; i\in [n].$ Let 
\begin{align*}
    \Lambda_1(\mathcal{A}) & =R_1(\mathcal{A}),\\ 
    \Lambda_i(\mathcal{A}) & =\sum_{i_2...i_m \in [i-1]^{m-1}} |a_{i i_2 ... i_m}| \left(\frac{\Lambda_{i_2} (\mathcal{A})}{|a_{i_2 i_2 ... i_2}|}\right)^{\frac{1}{m-1}} \cdots \left(\frac{\Lambda_{i_m} (\mathcal{A})}{|a_{i_m i_m ... i_m}|}\right)^{\frac{1}{m-1}}\\ 
    & \qquad + \sum_{i_2...i_m \notin [i-1]^{m-1}, (i_2 ... i_m)\neq (i ... i)} |a_{i i_2 ... i_m}|, \; i=2,3,...,n. 
\end{align*}
 A tensor $\mathcal{A}=(a_{i_1 i_2 ... i_m})\in \mathbb{R}^{[m,n]}$ is said to be a Nekrasov tensor if 
 \begin{equation}
     |a_{ii...i}| > \Lambda_i(\mathcal{A}), \; \forall\; i\in [n].
 \end{equation}
\end{defn}

\begin{theorem}\cite{ding2013m}\label{nonsingular H if and only if quasi-SDD}
A tensor $\mathcal{A}$ is a nonsingular $H$-tensor if and only if it is quasi-strictly diagonally dominant tensor.
\end{theorem}

\begin{theorem}\cite{zhang2018nekrasov}\label{nekrasov implies nonsingular H}
If a tensor $\mathcal{A}=(a_{i_1 i_2 ... i_m})\in \mathbb{C}^{[m,n]}$ is a Nekrasov tensor, then $\mathcal{A}$ is a nonsingular $H$-tensor.
\end{theorem}

\begin{theorem}\cite{ding2018p}\label{nonsingular H implies P}
A nonsingular $H$-tensor of even order with all positive diagonal entries is a $P$-tensor.
\end{theorem}

\begin{theorem}\cite{bai2016global}
For any $q\in \mathbb{R}^n$ and a $P$-tensor $\mathcal{A} \in \mathbb{R}^{[m,n]}$, the solution set of TCP$(\mathcal{A},q)$ is nonempty and compact.
\end{theorem}

\section{Main results}

%\begin{examp}\label{example of nekrasov Z tensor}
%We consider the Example 3.3 of \cite{zhang2018nekrasov}. Let $\mathcal{B}\in \mathbb{R}^{[4,4]}$ be such that $b_{1111}=8,\; b_{2222}=3.8,\; b_{3333}=3,\; b_{4444}=10,\; b_{1112}=b_{2111}=b_{1211}=b_{1121}=-1,\;$ $b_{3222}=b_{2322}=b_{2232}=b_{2223}=-1,$ $b_{4441}=b_{4414}=b_{4144}=b_{1444}=-3$ and all other entries of $\mathcal{B}$ are zeros. Then $R_1(\mathcal{B})=6,\; R_2(\mathcal{B})=4,\; R_3(\mathcal{B})=1,\; R_4(\mathcal{B})=9.$ Since $|b_{2222}=| < R_2(\mathcal{B}),$ so $\mathcal{B}$ is not diagonally dominant tensor. However $\lambda_1(\mathcal{B})=6,\; \lambda_2(\mathcal{B})=3.75,\;\lambda_3(\mathcal{B})=0.98,\;\lambda_4(\mathcal{B})=9.$ Also $|b_{ii...i}|>\lambda_i(\mathcal{B}),\; \forall\; i\in [4].$ Hence $\mathcal{B}$ is Nekrasov tensor. Also all the offdiagonal entries of $\mathcal{B}$ are nonpositive. Therefore $\mathcal{B}$ is a $Z$-tensor. Hence $\mathcal{B}$ is a Nekrasov $Z$-tensor.
%\end{examp}

We begin by introducing $B$-Nekrasov tensor. We want to define a class of tensors which will contain Nekrasov $Z$-tensors with positive diagonal entries. To do so we use a decomposition of tensors which will be useful in our discussion.\\

\NI Given a tensor $\mathcal{A}=(a_{i_1 i_2 ... i_m})\in \mathbb{R}^{[m,n]},$ we can write $\mathcal{A}$ as
\begin{equation}\label{decomposition equation}
    \mathcal{A} = \mathcal{B}^+ + \mathcal{C},
\end{equation}
where $\forall \; i\in [n]$ the elements of rowsubtensors of $\mathcal{B}$ and $\mathcal{C}$ are given by,
\begin{equation}\label{law for B+}
(R_i(\mathcal{B}^+))_{i_2 \cdots i_m} = a_{i i_2 ... i_m} - r_i^+, \;\forall \; i_l\in [n], \; l\in [m]
\end{equation}
and
\begin{equation}\label{law for C}
(R_i(\mathcal{C}))_{i_2 \cdots i_m} =r_i^+ , \;\forall \; i_l\in [n], \; l\in [m]..
\end{equation}
with
\begin{equation}\label{law for r_i^+}
    r_i^+ =\max_{i_2, ..., i_m \in [n], \; ( i_2 \cdots i_m) \neq (i \cdots i)} \{ 0, a_{i i_2 ... i_m} \}
\end{equation}

\begin{prop}
Let $\mathcal{A}=(a_{i_1 i_2 ... i_m})\in \mathbb{R}^{[m,n]}.$ If $\mathcal{A}$ is decomposed as given in the equation (\ref{decomposition equation}), then $\mathcal{B}^+$ is a $Z$-tensor and $\mathcal{C}$ is nonnegative tensor.
\end{prop}

Here we define $B$-Nekrasov tensor.
\begin{defn}
A tensor $\mathcal{A}\in \mathbb{R}^{[m,n]}$ is said to be a $B$-Nekrasov tensor if it can be written in the form of (\ref{decomposition equation}) with $\mathcal{B}^+$ as a Nekrasov $Z$-tensor whose diagonal entries are all positive.
\end{defn}

%\begin{remk}
%For this to happen for each row subtensor the maximum of the positive elements except the diagonal element of the tensor corresponding to that particular row subtensor must be strictly less than the corresponding diagonal element  of the original tensor.
%\end{remk}

Here is an example of a $B$-Nekrasov tensor which is not a Nekrasov tensor.%(nor $B$-tensor)

\begin{examp}
Let $\mathcal{A}=(a_{i j k l})\in \mathbb{R}^{[4,4]}$ be such that $a_{1111}=10,\; a_{2222}=6.8,\; a_{3333}=3,\; a_{4444}=10,\; a_{1112}=a_{1211}=a_{1121}=1,\; a_{1444}=-1$ and $a_{1jkl}=2$ for all other entries of $j,\; k,\; l\in [4],$ 
$a_{2111}=2,\; a_{2322}=a_{2232}=a_{2223}=2,\;$ and $a_{2jkl}=3$ for all other entries of $j,\; k,\; l\in [4],$
$a_{3222}=-1$ and $a_{3jkl}=0$ for all other entries of $j,\; k,\; l\in [4]$ $a_{4441}=a_{4414}=a_{4144}=-3,\;$ and $a_{4jkl}=0$ for all other entries of $j,\; k,\; l\in [4].$

\NI Then
\begin{equation}
    r_1^+ =\max \{0,1,-1,2\}=2,
\end{equation}
\begin{equation}
    r_2^+ =\max \{0,2,3\}=3,
\end{equation}
\begin{equation}
    r_3^+ =\max \{0,-1\}=0,
\end{equation}
\begin{equation}
    r_4+ =\max \{0,-3\}=0,
\end{equation}
Then the tensor can be decomposed into two tensor as $\mathcal{A}=\mathcal{B}^+ + \mathcal{C}$, where $(\mathcal{R}_1(\mathcal{C}))_{jkl}= r_1^+ =2$ forall $j,\; k,\; l\in [4],$ $(\mathcal{R}_2(\mathcal{C}))_{jkl}= r_2^+ =3$ forall $j,\; k,\; l\in [4]$ and all other entries of $\mathcal{C}$ are zeros. The tensor $\mathcal{B}^+$ is same as the tensor $\mathcal{B}$ of Example \ref{example of nekrasov Z tensor}, which is proved to be a Nekrasov $Z$-tensor. Also all the diagonal entries of $\mathcal{B}^+$ are $8,\; 3.8,\; 3,\; 10,$ which are positive. Therefore $\mathcal{B}^+$ is a Nekrasov $Z$-tensor with positive diagonal entries. Hence $\mathcal{A}$ is a $B$-Nekrasov tensor.

Now for $i=1$ we have $a_{1111}=10$ and $\Lambda_1(\mathcal{A})=\mathcal{R}_1(\mathcal{A}).$ Now
\begin{equation*}
    R_i(\mathcal{A})= \sum_{(i_2 ... i_m) \neq (i ... i)} |a_{i i_2 ... i_m}| = 122.
\end{equation*}
Clearly $10=a_{1111}< 122 = \Lambda_1(\mathcal{A}).$ Therefore $(\mathcal{A})$ is not a Nekrasov tensor.

%Then $R_1(\mathcal{B})=6,\; R_2(\mathcal{B})=4,\; R_3(\mathcal{B})=1,\; R_4(\mathcal{B})=9.$ Since $|b_{2222}=| < R_2(\mathcal{B}),$ so $\mathcal{B}$ is not diagonally dominant tensor. However $\lambda_1(\mathcal{B})=6,\; \lambda_2(\mathcal{B})=3.75,\;\lambda_3(\mathcal{B})=0.98,\;\lambda_4(\mathcal{B})=9.$ Also $|a_{ii...i}|>\lambda_i(\mathcal{A}),\; \forall\; i\in [4].$ Hence $\mathcal{B}$ is Nekrasov tensor. Also all the offdiagonal entries of $\mathcal{B}$ are nonpositive. Therefore $\mathcal{B}$ is a $Z$-tensor. Hence $\mathcal{A}$ is a $B$-Nekrasov tensor.
\end{examp}

\begin{prop}
A Nekrasov $Z$-tensor with positive diagonal entries is a $B$-Nekrasov tensor.
\end{prop}
\begin{proof}
The result immediately follows from the fact that for a Nekrasov $Z$-tensor with positive diagonal entries $\mathcal{A},$ the tensor can be easily decomposed into $\mathcal{A}=\mathcal{B}^+ + \mathcal{C},$ where $\mathcal{B}^+ = \mathcal{A}$ and $\mathcal{C} =\mathcal{O}.$
\end{proof}

Here we present a necessary and sufficient condition for $B$-Nekrasov tensor.

\begin{theorem}
Let $\mathcal{A}\in \mathbb{R}^{[m,n]}.$ Then $\mathcal{A}$ is a $B$-Nekrasov tensor if and only if the following condition hold:\\

(a) $\sum_{i_2...i_m\in [n]^{m-1}} a_{1i_2...i_m} >0 $ and $\frac{1}{n^{m-1}}\left( \sum_{i_2...i_m\in [n]^{m-1}} a_{1i_2...i_m} \right) > a_{1 j_2 \cdots j_m},$ for $j_2 \cdots j_m \in [n]^{m-1}$ and $(j_2, \cdots ,j_m)\neq (1,... ,1).$

(b) $r_i^+ <\min \left \{ \frac{\sum_{i_2...i_m\in [n]} a_{ii_2...i_m} \alpha_{i_2}\cdots \alpha_{i_m} + \sum_{i_2, ..., i_m \notin [i-1]^{m-1}} a_{i i_2 \cdots i_m}}{\sum_{i_2, ..., i_m \in [i-1]^{m-1}} \alpha_{i_2} \cdots \alpha_{i_m} + n^{m-1} -(i-1)^{m-1}},  a_{i...i} \right\}$, where $i\in [n]-\{1\},$
where $\alpha_{i_l}=\left(\frac{\Lambda_{i_l} (\mathcal{B}^+)}{|b^+_{i_l i_l ... i_l}|}\right)^{\frac{1}{m-1}},\; \forall\; 1\leq l \leq m$

\end{theorem}
\begin{proof}
Let conditions (a) and (b) hold for a tensor $\mathcal{A}.$ Let $\mathcal{A} = \mathcal{B}^+ + \mathcal{C}$ where $\mathcal{B}^+ $ and $\mathcal{C}$ are given by (\ref{decomposition equation}). We prove that the tensor $\mathcal{B^+}$ is a Nekrasov tensor.  Now
\begin{align}
   (0\leq)\; \Lambda_1(\mathcal{B}^+) &= \sum_{(i_2,...,i_m) \neq (i, ..., i)} |b_{1 i_2...i_m}^+|\notag \\
    &= \sum_{(i_2,...,i_m) \neq (i, ..., i)} |a_{1 i_2...i_m} - r_1^+| \notag \\
    &= \sum_{(i_2,...,i_m) \neq (i, ..., i)} (r_1^+ - a_{1 i_2...i_m}) \notag \\
    &= (n^{m-1} -1) r_1^+ - \sum_{(i_2,...,i_m) \neq (i, ..., i)} a_{1 i_2...i_m} \notag \\
    &= \left( n^{m-1} r_1^+ - \sum_{i_2,...,i_m \in [n]} a_{1 i_2...i_m} \right) + (a_{1 1 \cdots 1} - r_1^+).\label{comparison for lambda1}
\end{align}
Now from (a) $\sum_{i_2...i_m\in [n]^{m-1}} a_{1 i_2...i_m} >0 $ and $\frac{1}{n^{m-1}}\left( \sum_{i_2...i_m\in [n]^{m-1}} a_{1i_2...i_m} \right) > a_{1 j_2 \cdots j_m},$ for $j_2, \cdots ,j_m \in [n].$ Which implies
\begin{equation}\label{comparison with max and sum}
   r_1^+ = \max\{ 0, a_{1 i_2...i_m}\} < \frac{1}{n^{m-1}}\left( \sum_{i_2...i_m\in [n]^{m-1}} a_{1i_2...i_m} \right)
\end{equation}
Therefore,
\begin{equation}\label{first equation for necessary part}
    \left( n^{m-1}  r_1^+ - \sum_{i_2,...,i_m \in [n]} a_{1 i_2...i_m} \right) < 0.
\end{equation}
Now from (\ref{comparison for lambda1}) and (\ref{first equation for necessary part}) we obtain,
\begin{equation}
    \Lambda_1(\mathcal{B}^+) < a_{1 1 \cdots 1} - r_1^+ = |a_{1 1...1} - r_1^+| = |b_{1 \cdots 1}^+|.
\end{equation}
Thus for $i=1,$ the tensor $\mathcal{B}^+$ satisfies the condition of Nekrasov tensor.

\NI Now for $i= 2,3,..., n$, we have
\begin{align}
    0\leq \Lambda_i(\mathcal{B}^+) &=\sum_{i_2...i_m \in [i-1]^{m-1}} |b^+_{i i_2 ... i_m}| \left(\frac{\Lambda_{i_2} (\mathcal{B}^+)}{|b^+_{i_2 i_2 ... i_2}|}\right)^{\frac{1}{m-1}} \cdots \left(\frac{\Lambda_{i_m} (\mathcal{B}^+)}{|b^+_{i_m i_m ... i_m}|}\right)^{\frac{1}{m-1}} \notag \\
    & \qquad + \sum_{i_2...i_m \notin [i-1]^{m-1}, (i_2 ... i_m)\neq (i ... i)} |b^+_{i i_2 ... i_m}| \notag \\
    &= \sum_{i_2...i_m \in [i-1]^{m-1}} |a_{i i_2 \cdots i_m} - r_i^+| \left(\frac{\Lambda_{i_2} (\mathcal{B}^+)}{|b^+_{i_2 i_2 ... i_2}|}\right)^{\frac{1}{m-1}} \cdots \left(\frac{\Lambda_{i_m} (\mathcal{B}^+)}{|b^+_{i_m i_m ... i_m}|}\right)^{\frac{1}{m-1}} \notag \\
    & \qquad + \sum_{i_2...i_m \notin [i-1]^{m-1}, (i_2 ... i_m)\neq (i ... i)} |a_{i i_2 \cdots i_m} - r_i^+| \notag \\
    &= \sum_{i_2...i_m \in [i-1]^{m-1}} (r_i^+ - a_{i i_2 \cdots i_m}) \left(\frac{\Lambda_{i_2} (\mathcal{B}^+)}{|b^+_{i_2 i_2 ... i_2}|}\right)^{\frac{1}{m-1}} \cdots \left(\frac{\Lambda_{i_m} (\mathcal{B}^+)}{|b^+_{i_m i_m ... i_m}|}\right)^{\frac{1}{m-1}} \notag \\
    & \qquad + \sum_{i_2...i_m \notin [i-1]^{m-1}, (i_2 ... i_m)\neq (i ... i)} (r_i^+ - a_{i i_2 \cdots i_m}) \notag \\
    &= r_i^+ \sum_{i_2...i_m \in [i-1]^{m-1}} \alpha_{i_2} \cdots \alpha_{i_m} - \sum_{i_2...i_m \in [i-1]^{m-1}} a_{i i_2 \cdots i_m} \alpha_{i_2} \cdots \alpha_{i_m} \notag \\
    &\qquad + (n^{m-1} -(i-1)^{m-1} -1) r_i^+  - \sum_{i_2...i_m \notin [i-1]^{m-1}, (i_2 ... i_m)\neq (i ... i)} a_{i i_2 \cdots i_m} \notag \\
    &= \left[r_i^+ \left( \sum_{i_2...i_m \in [i-1]^{m-1}} \alpha_{i_2} \cdots \alpha_{i_m} + n^{m-1} -(i-1)^{m-1} \right)\right. \notag \\
    &\quad \left. - \sum_{i_2...i_m \in [i-1]^{m-1}} a_{i i_2 \cdots i_m} \alpha_{i_2} \cdots \alpha_{i_m} - \sum_{i_2...i_m \notin [i-1]^{m-1}} a_{i i_2 \cdots i_m} \right] \notag \\
    &\qquad \qquad - r_i^+ + a_{i \cdots i} \label{equation for necessary part of B Nekrasov}
\end{align}

\NI Now by (b) we have  $r_i^+ <\min \left \{ \frac{\sum_{i_2...i_m\in [n]} a_{ii_2...i_m} \alpha_{i_2}\cdots \alpha_{i_m} + \sum_{i_2...i_m \notin [i-1]^{m-1}} a_{i i_2 \cdots i_m}}{\sum_{i_2...i_m \in [i-1]^{m-1}} \alpha_{i_2} \cdots \alpha_{i_m} + n^{m-1} -(i-1)^{m-1}},  a_{i...i} \right\}$, where $i=2,3,...,n.$ This implies
\begin{align}
    r_i^+ \left( \sum_{i_2...i_m \in [i-1]^{m-1}} \alpha_{i_2} \cdots \alpha_{i_m} + n^{m-1} -(i-1)^{m-1} \right) &< \sum_{i_2...i_m \in [i-1]^{m-1}} a_{i i_2 \cdots i_m} \alpha_{i_2} \cdots \alpha_{i_m} \notag \\
    &\qquad + \sum_{i_2...i_m \notin [i-1]^{m-1}} a_{i i_2 \cdots i_m} \label{second equation for necessary part}
\end{align}
Now by (\ref{equation for necessary part of B Nekrasov}) and (\ref{second equation for necessary part}) we conclude that
\begin{equation}
    \Lambda_i(\mathcal{B}^+) < a_{i \cdots i} - r_i^+ = |a_{i \cdots i} - r_i^+|= |b_{i \cdots i}^+|, 
\end{equation}
Therefore $\mathcal{B}^+$ is a Nekrasov tensor. Note that by construction $\mathcal{B}^+$ is a $Z$-tensor and by the conditions (a) and (b), we have $b_{i...i}^+ =a_{i...i} -r_i^+>0\; \forall\; i\in [n].$ Therefore $\mathcal{B}^+$ is a Nekrasov $Z$-tensor with positive diagonal entries. Hence $\mathcal{A}$ is a $B$-Nekrasov tensor.

Conversely, let us assume that $\mathcal{A}$ is a $B$-Nekrasov tensor. We prove that (a) and (b) hold. Since $\mathcal{B}^+$ is a Nekrasov tensor with positive diagonal entries, we have
\begin{align*}
a_{1 1...1} - r_1^+ &= |a_{1 1...1} - r_1^+|\\
                    &> \sum_{(i_2,...,i_m) \neq (i, ..., i)} |a_{1 i_2...i_m} - r_1^+|\\
                    &= \sum_{(i_2,...,i_m) \neq (i, ..., i)} (r_1^+ - a_{1 i_2...i_m})\\
                    &= (n^{m-1} -1) r_1^+ - \sum_{(i_2,...,i_m) \neq (i, ..., i)} a_{1 i_2...i_m},
\end{align*}
and so,
\begin{equation}\label{comparison with max and sum}
   r_1^+ < \frac{1}{n^{m-1}}\left( \sum_{i_2...i_m\in [n]^{m-1}} a_{1i_2...i_m} \right).
\end{equation}
This implies $\sum_{i_2...i_m \in [n]^{m-1}} a_{1 i_2...i_m} >0$ and $\frac{1}{n^{m-1}}\left( \sum_{i_2...i_m\in [n]^{m-1}} a_{1i_2...i_m} \right) > a_{1 j_2 \cdots j_m},$ for $j_2 \cdots j_m \in [n]^{m-1}$ and $(j_2, \cdots ,j_m)\neq (1,... ,1).$ Hence condition (a) holds.

\NI Again since $\mathcal{B}^+$ is a Nekrasov tensor with positive diagonal entries, $|a_{i i...i} -r_i^+|$ $= a_{i i...i} -r_i^+$ and we also have,

\begin{align*}
     a_{i i...i} -r_i^+ &= |a_{i i...i} -r_i^+|\\
                        &> \Lambda_i(\mathcal{B}^+)\\
                        & = \sum_{i_2 ... i_m \in [i-1]^{m-1}} |a_{i i_2 \cdots i_m} - r_i^+| \alpha_{i_2}\cdots \alpha_{i_m} \\
                        & \qquad + \sum_{i_2 ... i_m \notin [i-1]^{m-1}, (i_2 ... i_m)\neq (i ... i)} |a_{i i_2 \cdots i_m} - r_i^+|\\
                        &= \sum_{i_2 ... i_m \in [i-1]^{m-1}} (r_i^+ - a_{i i_2 \cdots i_m})\alpha_{i_2}\cdots \alpha_{i_m} \\
                        & \qquad + \sum_{i_2 ... i_m \notin [i-1]^{m-1}, (i_2 ... i_m)\neq (i ... i)} (r_i^+ - a_{i i_2 \cdots i_m})\\
                        &= r_i^+ \sum_{i_2 ... i_m \in [i-1]^{m-1}} \alpha_{i_2} \cdots \alpha_{i_m} - \sum_{i_2 ... i_m \in [i-1]^{m-1}} a_{i i_2 \cdots i_m} \alpha_{i_2} \cdots \alpha_{i_m} \\
                       &\qquad + \left(n^{m-1} -(i-1)^{m-1} -1\right) r_i^+  - \sum_{i_2 ... i_m \notin [i-1]^{m-1}, (i_2 ... i_m)\neq (i ... i)} a_{i i_2 \cdots i_m}.
\end{align*}
Then
\begin{align*}\label{second equation for sufficient part}
    r_i^+ \left( \sum_{i_2, ..., i_m \in [i-1]^{m-1}} \alpha_{i_2} \cdots \alpha_{i_m} + n^{m-1} -(i-1)^{m-1} \right) &< \sum_{i_2, ..., i_m \in [i-1]^{m-1}} a_{i i_2 \cdots i_m} \alpha_{i_2} \cdots \alpha_{i_m}\\
    &\qquad + \sum_{i_2, ..., i_m \notin [i-1]^{m-1}} a_{i i_2 \cdots i_m}.
\end{align*}
Therefore,
\begin{equation*}
    r_i^+ < \frac{\sum_{i_2...i_m\in [n]} a_{ii_2...i_m} \alpha_{i_2}\cdots \alpha_{i_m} + \sum_{i_2, ..., i_m \notin [i-1]^{m-1}} a_{i i_2 \cdots i_m}}{\sum_{i_2, ..., i_m \in [i-1]^{m-1}} \alpha_{i_2} \cdots \alpha_{i_m} + n^{m-1} -(i-1)^{m-1}}
\end{equation*}
and so (b) holds.
\end{proof}

Here we propose a weak restriction on a Nekrasov tensor to costruct a positive diagonal matrix $W$ which scales the Nekrasov tensor and transforms it to a strictly diagonally dominated tensor.

\begin{theorem}\label{theorem to convert into SDD}
Let for even $m,$ $\mathcal{B}=(b_{i_1 i_2 ... i_m})\in \mathbb{R}^{[m,n]}$ be a Nekrasov tensor such that, for each $i=1,2,..., n-1,\; b_{i i_2 ... i_m} \neq 0$ for some $i_2, ..., i_m > i.$ Then the matrix $W=diag(w_1, ..., w_n),$ with $w_i= \left( \frac{\Lambda_i(\mathcal{B})}{|b_{i...i}|} \right)^{\frac{1}{m-1}}$ for $i\in[n-1]$ and $w_n= \left( \frac{\Lambda_n(\mathcal{B})}{|b_{n...n}|} \right)^{\frac{1}{m-1}} + \epsilon,$ $\epsilon \in \left( 0, 1- \left( \frac{\Lambda_n(\mathcal{B})}{|b_{n...n}|} \right)^{\frac{1}{m-1}} \right),$ has positive diagonal entries and the tensor $\mathcal{A} = \mathcal{B} W$ is strictly diagonally dominated tensor.
\end{theorem}
\begin{proof}
Note that, by the hypothesis, for each $i\in [n-1],$ $b_{i i_2 ... i_m} \neq 0$ for some $i_2, ..., i_m > i.$ Then $\Lambda_i(\mathcal{B}) >0$ for $i\in [n-1]$ and so the diagonal entries of the matrix $W$ are all positive. Let $W=diag(w_1, ..., w_n).$ Now for $i\in [n-1]$ we have,
\begin{align}
    |b_{i...i}|w_i^{m-1} &=\Lambda_i(\mathcal{B}) \notag \\
                         &=\sum_{i_2...i_m \in [i-1]^{m-1}} |b_{i i_2 ... i_m}| \left(\frac{\Lambda_{i_2} (\mathcal{B})}{|b_{i_2 i_2 ... i_2}|}\right)^{\frac{1}{m-1}} \cdots \left(\frac{\Lambda_{i_m} (\mathcal{B})}{|b_{i_m i_m ... i_m}|}\right)^{\frac{1}{m-1}} \notag \\
                         & \qquad + \sum_{(i_2 ... i_m) \in [n-1]^{m-1} \backslash\ [i-1]^{m-1}, (i_2 ... i_m)\neq (i ... i)} |b_{i i_2 ... i_m}| \notag \\
                          &=\sum_{i_2...i_m \in [i-1]^{m-1}} |b_{i i_2 ... i_m}| \left(\frac{\Lambda_{i_2} (\mathcal{B})}{|b_{i_2 i_2 ... i_2}|}\right)^{\frac{1}{m-1}} \cdots \left(\frac{\Lambda_{i_m} (\mathcal{B})}{|b_{i_m i_m ... i_m}|}\right)^{\frac{1}{m-1}} \notag \\
                         & \qquad + \sum_{(i_2 ... i_m) \in [n-1]^{m-1} \backslash\ [i-1]^{m-1}, (i_2 ... i_m)\neq (i ... i)} |b_{i i_2 ... i_m}| \notag \\
                         & \qquad \qquad + \sum_{(i_2 ... i_m) \in [n]^{m-1} \backslash\ [n-1]^{m-1}}|b_{i i_2 ... i_m}| \label{main equation for scalling}
\end{align}

\NI Since $\mathcal{B}$ is a Nekrasov tensor, $\Lambda_i(\mathcal{B}) < |b_{ii...i}|$ and so $w_i <1$ for all $i\in[n-1].$ Therefore 
\begin{equation}\label{2nd main equation for scalling}
    |b_{i i_2...i_m} \geq |b_{i i_2 ... i_m}| w_{i_2}\cdots w_{i_m}, \; (i_2 ... i_m) \in ([n-1]^{m-1} \backslash\ [i-1]^{m-1}), (i_2 ... i_m)\neq (i ... i).
\end{equation}

\NI Also by the assumption, $\epsilon < 1- \left( \frac{\Lambda_n(\mathcal{B})}{|b_{n...n}|} \right)^{\frac{1}{m-1}} $ an so $w_n <1.$ Therefore
\begin{equation}\label{3rd main equation for scalling}
    |b_{i i_2...i_m}| \geq |b_{i i_2 ... i_m}| w_{i_2}\cdots w_{i_m}, \text{ for } (i_2 ... i_m) \in [n]^{m-1} \backslash\ [n-1]^{m-1}.
\end{equation}
From (\ref{main equation for scalling}), (\ref{2nd main equation for scalling}) and (\ref{3rd main equation for scalling}) we deduce
\begin{equation}\label{4th main equation for scalling}
    |b_{i i ... i}| w_{i}\cdots w_{i} \geq \sum_{(i_2 ... i_m) \neq (i i ... i)} |b_{i i_2 ... i_m}| w_{i_2}\cdots w_{i_m}.
\end{equation}
Now, if either (\ref{2nd main equation for scalling}) or (\ref{3rd main equation for scalling}) are strict inequalities,  then (\ref{4th main equation for scalling}) is also strict.

\NI If $b_{i i_2 ... i_m}\neq 0,$ for $(i_2 ... i_m) \in [n]^{m-1} \backslash\ [n-1]^{m-1}$, since $w_n <1,$ then $|b_{i i_2 ... i_m}| w_{i_2}\cdots w_{i_m} < |b_{i i_2 ... i_m}|$ and so (\ref{3rd main equation for scalling}) and (\ref{4th main equation for scalling}) are strict.

\NI If $b_{i i_2 ... i_m}= 0,$ the by the assumption $b_{i i_2 ... i_m} \neq0$ for some $(i_2 ... i_m) \in ([n-1]^{m-1} \backslash\ [i-1]^{m-1}), (i_2 ... i_m)\neq (i ... i)$ Since $w_i <1,$ then $|b_{i i_2 ... i_m}| w_{i_2}\cdots w_{i_m} < |b_{i i_2 ... i_m}|$ and so $(\ref{2nd main equation for scalling})$ and (\ref{4th main equation for scalling}) are strict. Therefore the condition of strictly diagonally dominance for the tensor $\mathcal{A}W$ holds for $i\in [n-1].$

Finally, since $\epsilon |b_{n ...n}| >0,$ we have
\begin{align*}
    |b_{n ...n} w_n w_n \cdots w_n| &= \Lambda_n(\mathcal{B}) + \epsilon |b_{n ...n}|\\
    &= \sum_{i_2...i_m \in [n-1]^{m-1}} |b_{n i_2 ... i_m}| \left(\frac{\Lambda_{i_2} (\mathcal{B})}{|b_{i_2 i_2 ... i_2}|}\right)^{\frac{1}{m-1}} \cdots \left(\frac{\Lambda_{i_m} (\mathcal{B})}{|b_{i_m i_m ... i_m}|}\right)^{\frac{1}{m-1}}\\
    & \qquad +\sum_{(i_2 ... i_m)\in [n]^{m-1} \backslash\ [n-1]^{m-1}, (i_2 ... i_m)\neq (n...n)}|b_{n i_2 ... i_m}|  + \epsilon |b_{n ...n}|\\
    &\geq \sum_{(i_2 ... i_m) \neq (n...n)} |b_{n i_2 ... i_m}| w_{i_2}\cdots w_{i_m} + \epsilon |b_{n ...n}| \text{ using (\ref{3rd main equation for scalling}) }\\
    &> \epsilon |b_{n ...n}|
\end{align*}
and so the last condition of strictly diagonally dominance for the tensor $\mathcal{A}W$ holds for $i=n.$
\end{proof}

\begin{theorem}
Let for even $m,$ $\mathcal{A}=(a_{i_1 i_2 ... i_m})\in \mathbb{R}^{[m,n]},$ be a $B$-Nekrasov tensor such that for each $i=1,...,n-1,\; \exists\; i_2, ..., i_m >i$ with $a_{ii_2 ... i_m }< \max \{0, a_{i i_2 ... i_m}: (i_2 ... i_m) \neq (i ... i)\}= r_i^+.$ Let $\mathcal{B}^+$ be the tensor given by equation (\ref{decomposition equation}). Then the matrix $W=diag(w_1 ,\cdots w_n)$ with $w_i= \left(\frac{\Lambda_i(\mathcal{B}^+)}{a_{ii ... i} - r_i^+}\right)^{\frac{1}{m-1}},$ for $i\in [n-1]$ and $w_n = \left(\frac{\Lambda_n(\mathcal{B}^+)}{a_{nn ... n} - r_n^+}\right)^{\frac{1}{m-1}} + \epsilon,$ $\epsilon \in \left( 0, 1- \left(\frac{\Lambda_n(\mathcal{B}^+)}{a_{nn ... n} - r_n^+}\right)^{\frac{1}{m-1}} \right),$ has positive diagonal entries with each entry less than $1$ and the tensor $\mathcal{B}^+ W$ is a strictly diagonally dominant $Z$-tensor.
\end{theorem}
\begin{proof}
Note that $\mathcal{B}^+$ is a Nekrasov tensor by the definition of a $B$-Nekrasov tensor. Since $a_{ii_2 ... i_m }< r_i^+,\;\forall\; i\in [n],$ we conclude that $\mathcal{B}^+$ satisfies the conditions of the Theorem \ref{theorem to convert into SDD}. Therefore by Theorem \ref{theorem to convert into SDD}, we conclude that $\mathcal{B}^+ W$ is strictly diagonally dominant tensor. Also $\mathcal{B}^+$ is a $Z$-tensor since $B$ is a $B$-Nekrasov tensor. Therefore $\mathcal{B}^+W$ is a strictly diagonally dominant $Z$-tensor. 
\end{proof}

\begin{prop}
Let $\mathcal{A}$ be a $B$-Nekrasov tensor of even order and $\mathcal{A}$ is decomposed as $\mathcal{B}^+ + \mathcal{C}.$ If $\mathcal{C}=O,$ then $\mathcal{A}$ is a nonsingular $H$-tensor.
\end{prop}
\begin{proof}
If $\mathcal{C}=O,$ then $\mathcal{A}=\mathcal{B}^+.$ Since $\mathcal{A}$ is a $B$-Nekrasov tensor, $\mathcal{B}^+$ is a Nekrasov $Z$-tensor with positive diagonal entries. This implies $\mathcal{B}^+$ is a Nekrasov tensor. Hence by Theorem \ref{nekrasov implies nonsingular H}, $\mathcal{B}^+=\mathcal{A}$ is a nonsingular $H$-tensor.
\end{proof}

\begin{lemma}
Let $\mathcal{B}, \; \mathcal{C}\in \mathbb{R}^{[m,n]},$ $m$ be even. If $\mathcal{B}$ is a $P$-tensor and $\mathcal{C}$ is a nonnegative tensor with constant row subtensors, then $\mathcal{B} +\mathcal{C}$ is a $P$-tensor.
\end{lemma}
\begin{proof}
Let $\mathcal{B},\; \mathcal{C}\in \mathcal{R}^{[m,n]}$ with the condition that $\mathcal{B}$ is a $P$-tensor and $\mathcal{C}$ is a nonnegative tensor with constant row subtensors. Then $\forall \;0\neq v \in \mathbb{R}^n \; \exists \; i\in [n]$ such that
\begin{equation}
    v_i (\mathcal{B}v^{m-1})_i >0.
\end{equation}
Let $\mathcal{C}= (c_{i_1 i_2 ... i_m})$ be a nonnegative tensor with costant row subtensors. Let the row subtensors of $\mathcal{C}$ are $\mathcal{R}_i(\mathcal{C})\; \forall \; i\in [n].$ Then for each $i\in [n],$ $(\mathcal{R}_i(\mathcal{C}))_{i_2 ... i_m} = c_i,\;\forall \; i_2, ..., i_m \in [n],$ for some $c_i \geq 0.$

\NI If possible let $\mathcal{A}= \mathcal{B}+\mathcal{C}$ is not a $P$-tensor. Then $\exists \; 0\neq v\in \mathbb{R}^n$ such that $v_i (\mathcal{C} v^{m-1})_i \leq 0,\; \forall \; i\in [n].$ This implies for all $i\in [n],$
\begin{align}
   & v_i (\mathcal{B} v^{m-1})_i + v_i (\mathcal{C} v^{m-1})_i \leq 0 \notag \\
   \implies & v_i (\mathcal{B} v^{m-1})_i + v_i c_i (\sum_{i=1}^n v_i )^{m-1} \leq 0 \notag \\
   \implies & \frac{v_i^{m-2}}{(\tau^{m-1})^2}\left[ v_i (\mathcal{B} v^{m-1})_i + v_i c_i \tau^{m-1} \right] \leq 0 \text{ where, } \tau=(\sum_{i=1}^n v_i ) \neq0, \notag \\
   \implies & \left(\frac{v_i}{\tau}\right)^{m-1} \left(\mathcal{B} \left(\frac{v}{\tau}\right)^{m-1}\right)_i +  \left(\frac{v_i}{\tau}\right)^{m-1} c_i \leq 0 \notag \\
   \implies & y_i (\mathcal{A} y^{m-1})_i + y_i c_i \leq 0, \text{ where } y=\frac{1}{\tau} v\label{last equation to prove P tensor}.
\end{align}

\NI Now we prove that for all $i\in [n],$ $y_i (\mathcal{A} y^{m-1})_i + y_i c_i \leq 0$ for arbitrary $c_i\geq0$ implies $y_i (\mathcal{A} y^{m-1})_i \leq 0.$
If not, then $\exists \; k\in [n]$ such that $y_k (\mathcal{A} y^{m-1})_k >0.$ Now if $y_k \geq 0$ then $y_k (\mathcal{A} y^{m-1})_k + y_k r_k >0.$ Again, if $y_k <0$ then for $\epsilon >0$ such that $y_k (\mathcal{A} y^{m-1})_k -\epsilon >0$ we can choose $r_k = \frac{y_k (\mathcal{A} y^{m-1})_k -\epsilon}{|y_k|} >0.$ Then $y_k (\mathcal{A} y^{m-1})_k + y_k r_k = \epsilon >0.$ This contradicts the assumption. Thus (\ref{last equation to prove P tensor}) implies $y_i (\mathcal{A} y^{m-1})_i \leq 0,\; \forall\; i\in [n].$ Therefore $\exists \; 0\neq y\in \mathbb{R}^n$ such that $y_i (\mathcal{A} y^{m-1})_i \leq 0,\; \forall\; i\in [n].$ This implies $\mathcal{A}$ is not a $P$-tensor.
\end{proof}

\begin{theorem}
If $\mathcal{A}$ is a $B$-Nekrasov tensor of even order, then $\mathcal{A}$ is a $P$-tensor.
\end{theorem}
\begin{proof}
By the definition of $B$-Nekrasov tensor $\mathcal{A}$ can be decomposed as $\mathcal{A}= \mathcal{B}^+ + \mathcal{C},$ where $\mathcal{B}^+$ is a Nekrasov $Z$-tensor whose all diagonal entries are positive and $\mathcal{C}$ is a nonnegative tensor with constant row subtensors. Now by Theorem \ref{1st theorem of Nekrasov Z tensor} we conclude that $\mathcal{B}^+$ is a $P$-tensor. Therefore the result follows from the fact that $\mathcal{A}$ is the sum of a $P$-tensor and a nonnegative tensor with constant row subtensors.
\end{proof}

\begin{remk}
The $P$-tensors play an important role in tensor complementarity theory. We know that the SOL$(\mathcal{A}, q)$ is nonempty and compact if the involved tensor $\mathcal{A}$ is a $P$-tensor. Here we prove that a $B$-Nekrasov tensor of even order $\mathcal{A}$ is a $P$-tensor. Hence we conclude that for a even order $B$-Nekrasov tensor $\mathcal{A},$ the solution set of TCP$(\mathcal{A}, q)$ is nonempty and compact.
\end{remk}

%\begin{prop}
%If $\mathcal{A}$ is a symmetric $B$-Nekrasov tensor then 
%\end{prop}

\section{Conclusion}
In  this article, we introduce $B$-Nekrasov tensor. We show that the class of $B$-Nekrasov tensor contains Nekrasov $Z$-tensor with positive diagonal entries. We present a necessary and sufficient condition for a $B$-Nekrasov tensor. We prove that the class of $B$-Nekrasov tensor is a subclass of $P$-tensor.

\section{Acknowledgment}
The author R. Deb is thankful to the Council of Scientific $\&$ Industrial Research (CSIR), India, Junior Research Fellowship scheme for financial support.

\bibliographystyle{plain}
\bibliography{referencesall}

\end{document}